\documentclass[11pt]{amsart}
\usepackage{geometry}               
\usepackage{graphicx}
\usepackage{tikz}
\usepackage{amssymb}
\usepackage{pdfsync}
\usepackage{hyperref}

\usepackage{epstopdf}
\usepackage{mathtools}
\usepackage{braket}
\usepackage{booktabs}

\mathtoolsset {showonlyrefs=true} 
\newtheorem{theorem}{Theorem}

\newtheorem{definition}[theorem]{Definition}
\newtheorem{remark}[theorem]{Remark}
\newtheorem{proposition}[theorem]{Proposition}
\newtheorem{lemma}[theorem]{Lemma}

\newcommand{\R}{\mathbb{R}}
\newcommand{\Mrom}{\mathrm{M}}
\newcommand{\SSS}{\mathbb{S}}
\newcommand{\N}{\mathbb{N}}
\newcommand{\Z}{\mathbb{Z}}

\newcommand{\fbold}{\boldsymbol{f}}
\newcommand{\fboldstar}{{\boldsymbol{f}_{\!\star}}}
\newcommand{\fboldbot}{\boldsymbol{f}_{\!\bot}}
\newcommand{\gbold}{\boldsymbol{g}}
\newcommand{\hbold}{\boldsymbol{h}}
\newcommand{\mbold}{\boldsymbol{m}}
\newcommand{\ubold}{\boldsymbol{u}}
\newcommand{\obold}{\boldsymbol{0}}
\newcommand{\alphabold}{\boldsymbol{\alpha}}

\newcommand{\Fhat}{\hat{F}}
\newcommand{\Ghat}{\hat{G}}
\newcommand{\Hhat}{\hat{H}}

\newcommand{\QuantumN}{\mathrm{N}}

\newcommand{\SO}{\mathrm{SO}}
\newcommand{\norm}[1]{\left\lVert #1\right\rVert}
\newcommand{\abs}[1]{\left\lvert#1\right\rvert}
\newcommand{\Round}[1]{\left(#1\right)}

\title[A sharpened energy-Strichartz inequality]{A sharpened energy-Strichartz inequality for the wave equation}

\author{Giuseppe Negro}

\begin{document}
\begin{abstract}
    We consider the sharp Strichartz estimate for the wave equation on $\mathbb R^{1+5}$ in the energy space, due to Bez and Rogers. We show that it can be refined by adding a term proportional to the distance from the set of maximisers, in the spirit of the classical sharpened Sobolev estimate of Bianchi and Egnell.
\end{abstract}

\maketitle

\section{Introduction}
Let $\mathcal{H}$ denote the \emph{energy space} for the wave equation; precisely, $\mathcal H$ is the real Hilbert space obtained as the completion of the Schwartz space with scalar product 
\begin{equation}\label{eq:PhaseFive_scalprod}
	\Braket{\fbold|\gbold}_{\mathcal{H}}:= \int_{\R^5} \nabla f_0\cdot \nabla g_0\, dx + \int_{\R^5} f_1 g_1\, dx,
\end{equation}
where $\fbold=(f_0, f_1), \gbold=(g_0, g_1)$. Bez and Rogers \cite{BezRogers13} proved the following sharp inequality: for all $u\colon \R^{1+5}\to \R$ satisfying $u_{tt}=\Delta u$, letting $\ubold(0)=(u(0), u_t(0))$, it holds that
\begin{equation}\label{eq:BezRogers}
	\norm{u}_{L^4(\R^{1+5})}^2\le \frac{1}{8\pi}\norm{\ubold(0)}_{\mathcal{H}}^2.
\end{equation}
Moreover, there is equality in \eqref{eq:BezRogers} if and only if
\begin{equation}\label{eq:MromFive_Manifold}
	\begin{split} 
	\ubold(0)\in\Mrom :=\Set{ c\Gamma_{\!\alphabold} \fboldstar |c\in\R,\  \alphabold\in
		\SSS^1 \times \R^7 \times \SO(5)
		},
	\end{split}
\end{equation}
where $\fboldstar:=( 4 (1+\abs{\cdot}^2)^{-2}, 0)$, and $\Gamma_{\!\alphabold}$ is a certain representation of the natural symmetry group of~\eqref{eq:BezRogers}; we will define this operator in the following section. The distance from $\Mrom$ is
\begin{equation}\label{eq:distance_fived}
	\mathrm{dist}(\fbold,\Mrom):= \inf \Set{ \norm{\fbold-c\Gamma_{\!\alphabold} \fboldstar}_{\mathcal{H}} |   c\in\R,\  \alphabold\in
		\SSS^1 \times \R^7 \times \SO(5) }.
\end{equation}
In this note we will prove that~\eqref{eq:BezRogers} can be sharpened, by adding a term that is proportional to such distance.
\begin{theorem}\label{thm:main_fived}
	There exists an absolute constant $C>0$ such that, for all $u\colon\R^{1+5}\to \R$ satisfying $u_{tt}=\Delta u$ and $\ubold(0)\in \mathcal{H}$, it holds that
  \begin{equation}\label{eq:main_fived}
  	C\mathrm{dist}(\ubold(0), \Mrom)^2\le \frac{1}{8\pi}\norm{\ubold(0)}_{\mathcal{H}}^2-\norm{u}_{L^4(\R^{1+5})}^2 \le   \frac{1}{8\pi}\mathrm{dist}(\ubold(0), \Mrom)^2.
\end{equation} 
\end{theorem}

This theorem is analogous to~\cite[Theorem~1.1]{Ne18}, which is a sharpening of the conformal Strichartz estimate of Foschi~\cite{Foschi}, based on the computation of a spectral gap using the Penrose conformal compactification of the Minkowski spacetime. In applying the same approach here, we will face a fundamental new difficulty, related to the lack of conformal invariance of~\eqref{eq:BezRogers}. The computation of the spectral gap is more involved: both the relevant quadratic forms and the scalar product of $\mathcal{H}$ are not diagonal in their spherical harmonics expansions. This requires the introduction of several new ingredients. We find it remarkable that a method based on conformal transformations also works in the present non conformally-invariant case. 

The upper bound in Theorem~\ref{thm:main_fived} follows from a general, abstract argument. The lower bound, on the other hand, is much more delicate and it will be obtained by the method of Bianchi--Egnell~\cite{BiEg91}. More precisely, we will obtain a local version of Theorem~1, with an effective constant for the lower bound. This local result will then be made global via the profile decomposition of Bahouri and G\'erard~\cite{BaGe99}, following the same steps as in the aforementioned~\cite{Ne18}. We refer to~\cite{DuMeRo11, Ne18} for further background and references.

Throughout all of this note, we only consider real-valued solutions to the wave equation. This is a natural assumption; for more details, see~\cite[Remark~3.2]{Ne18}.

In the next section, we will prove Theorem~\ref{thm:main_fived}. The necessary computations with spherical harmonics are collected in the appendix.
\section{Proof of Theorem~\ref{thm:main_fived}}
The right-hand inequality in \eqref{eq:main_fived} is an immediate consequence of~\cite[Proposition~2.1]{Ne18}. We will thus focus on the following theorem, a local version of Theorem~\ref{thm:main_fived}. Once this is proved, the global version will be a consequence of the profile decomposition of Bahouri and G\'erard~\cite{BaGe99}, applying verbatim the argument in~\cite[§~5]{Ne18}. We omit the details.
\begin{theorem}\label{thm:main_local}
  For all $u\colon\mathbb R^{1+5}\to \mathbb R$ satisfying $u_{tt}=\Delta u$ and $\mathrm{dist}(\ubold(0), \Mrom)< \norm{\ubold(0)}_{\mathcal{H}},$
  \begin{equation}\label{eq:prop_main_fived}
\frac{36}{85}\frac{1}{8\pi}\mathrm{dist}(\ubold(0), \Mrom)^2 + O(\mathrm{dist}(\ubold(0),\Mrom)^3)
 \le \frac{1}{8\pi}\norm{\ubold(0)}_{\mathcal{H}}^2-\norm{u}_{L^4(\R^{1+5})}^2.
 \end{equation}
\end{theorem}
Before we can begin with the proof, we give the precise definition of the symmetry operators $\Gamma_{\!\alphabold}$:
\begin{equation}\label{eq:GcalFiveAction}
	\Gamma_{\!\alphabold}\fbold:=
			R(t_0\sqrt{-\Delta}+\theta)
			\begin{bmatrix}
			            e^{\frac32 \sigma} f_0\Round{e^\sigma A(\cdot + x_0)} \\  
			            e^{\frac52 \sigma} f_1\Round{e^\sigma A(\cdot + x_0)}
			 \end{bmatrix}.
\end{equation}
Here, the matrix-valued operator $R$ is given by
\begin{equation}\label{eq:R_op}
R(\cdot):=
\begin{bmatrix} 
        \cos(\cdot ) & \frac{\sin( \cdot )}{\sqrt{-\Delta}}\\ 
    -\sqrt{-\Delta}\sin(\cdot) & \cos(\cdot)
\end{bmatrix},
\end{equation}
and the parameter $\alphabold$ is
\begin{equation}\label{eq:alpha_five}
	\begin{array}{cc}
		\alphabold=(t_0, \theta, \sigma, x_0, A), & t_0\in \R,\ \theta\in\SSS^1,\ \sigma \in\R,\ x_0\in \R^5,\ A\in \SO(5).
	\end{array}
\end{equation}
\begin{remark} This definition is analogous to~\cite[eq.~(8)]{Ne18}, which concerns Foschi's sharp conformal Strichartz estimate, mentioned in the introduction. However, here there are no Lorentz boosts. This is a first manifestation of the lack of full invariance of~\eqref{eq:BezRogers} under conformal transformations of Minkowski spacetime.
\end{remark}
\begin{remark}\label{rem:new_maximizer}
    In the aforementioned~\cite[Corollary~1.2]{BezRogers13}, Bez and Rogers actually considered $(0, (1+\abs{\cdot}^2)^{-3})$ instead of $\fboldstar$. Both belong to the orbit $\Mrom$; letting $\alphabold=(0, \pi/2, 0,0,0)$, we have that $-2^{-2}\Gamma_{\!\alphabold}\fboldstar=(0, (1+\abs{\cdot}^2)^{-3})$. Thus, the choice of which maximiser to consider is inessential. 
\end{remark}
These operators $\Gamma_{\!\alphabold}$ preserve both sides of the Strichartz inequality~\eqref{eq:BezRogers}. Precisely, we introduce the notation 
\begin{equation}\label{eq:def_wave_propagator}
    S\fbold(t,x):=\cos(t\sqrt{-\Delta})f_0(x) + \frac{\sin(t\sqrt{-\Delta})}{\sqrt{-\Delta}} f_1(x), 
\end{equation}
for the solution operator to the wave equation $u_{tt}=\Delta u$ with initial data $\ubold(0)=\fbold\in\mathcal{H}$,  and we have that (see~\cite[eq.~(9)]{Ne18})
\begin{equation}\label{eq:five_dim_unitary}
	\begin{array}{cc}
	\norm{\Gamma_{\!\alphabold} \fbold}_{\mathcal{H}} =\norm{\fbold}_{\mathcal{H}}, & \norm{S\Gamma_{\!\alphabold} \fbold}_{L^4(\R^{1+5})}=\norm{S\fbold}_{L^4(\R^{1+5})}.
	\end{array}
\end{equation}
With this newly introduced notation, we can denote the right-hand side of the sought inequality~\eqref{eq:prop_main_fived}, which is known as the \emph{deficit functional} of~\eqref{eq:BezRogers}, as follows:
\begin{equation}\label{eq:deficit_funcional_def}
    \Phi(\fbold):=\frac{1}{8\pi}\norm{\fbold}_{\mathcal{H}}^2-\norm{S \fbold}_{L^4(\R^{1+5})}^2.
\end{equation}
We can now begin the proof of Theorem~\ref{thm:main_local}. This will occupy the rest of the section.

For $\fbold\in\mathcal{H}\setminus\Set{\obold}$, the invariance of $\Phi$ under the operators $\Gamma_{\!\alphabold}$ yields $\Phi(\fbold)=\Phi(c\fboldstar + \fboldbot)$, where $c\ne 0$ and
\begin{equation}\label{eq:min_probl_result_fived}
	\begin{array}{ccc}
		\norm{\fboldbot}_{\mathcal{H}}=\mathrm{dist}(\fbold, \Mrom),&\text{and} &\fboldbot \bot T_{\fboldstar} \Mrom;
	\end{array}
\end{equation}
see the proof of~\cite[Proposition~5.3]{Ne18}. Here $\bot$ denotes orthogonality with respect to $\Braket{\cdot|\cdot}_{\mathcal H}$, and the tangent space is given by
\begin{equation}\label{eq:tangent_space_Mrom_fived}
	T_\fboldstar \Mrom := \mathrm{span} \Set{ \fboldstar, \left.\nabla_{\!\alphabold} \Gamma_{\!\alphabold} \fboldstar\right|_{\alphabold=\boldsymbol{0}}}, 
\end{equation}
where $\nabla_{\!\alphabold}$ is the list of derivatives with respect to all parameters \eqref{eq:alpha_five}. We will give a precise description of such tangent space below. 

Obviously, $\Phi(\fboldstar)=0$. Moreover, $\fboldstar$ is a critical point of $\Phi$, being indeed a global minimiser. The expansion of $\Phi$ to second order thus reads
\begin{equation}\label{eq:psi_taylor_fived}
    \Phi(\fbold)=\Phi(c\fboldstar+\fboldbot)=
    Q(\fboldbot, \fboldbot) + O(\lVert \fboldbot\rVert_{\mathcal H}^3),  
\end{equation}
where $Q$ is a quadratic form, whose precise expression we will give in the forthcoming~\eqref{eq:general_quadratic_form_fived}. We are thus reduced to prove the following coercivity; 
\begin{equation}\label{eq:coercivity_target}
    \begin{array}{cc}    
        Q(\fboldbot, \fboldbot)\ge \frac{36}{85}\frac{1}{8\pi}\lVert \fboldbot\rVert_{\mathcal H}^2, & \text{for all }\fboldbot\bot T_{\fboldbot}\Mrom.
    \end{array}
\end{equation}

As mentioned in the introduction, the scalar product $\Braket{\cdot|\cdot}_\mathcal{H}$ and the associated orthogonality relation $\bot$ are cumbersome to work with; see the forthcoming Remark~\ref{rem:cumbersome_orthogonality} in the appendix. The following lemma shows that we are free to modify $\bot$, leveraging on the fact that $\fboldstar$ is a maximiser for~\eqref{eq:BezRogers}, hence a minimiser for the deficit functional $\Phi$. 
\begin{lemma}\label{lem:change_ortho}
    Let $\widetilde{\Braket{\cdot|\cdot}}$ be a scalar product on $\mathcal{H}$, and denote by $\widetilde{\bot}$ the corresponding orthogonality relation. If there is a $C_\star>0$ such that $Q(\gbold, \gbold)\ge C_\star\lVert \gbold\rVert_{\mathcal H}^2$ for all $\gbold\widetilde{\bot}T_\fboldstar \Mrom$, then $Q(\fboldbot, \fboldbot)\ge C_\star\lVert \fboldbot\rVert_{\mathcal{H}}^2$ for all $\fboldbot \bot T_{\fboldstar}\Mrom$.    
\end{lemma}
\begin{proof}
    Let $\fboldbot\bot T_{\fboldstar}\Mrom$ and decompose it as $\fboldbot=\gbold+\hbold$, where $\gbold\widetilde{\bot}T_{\fboldstar}\Mrom$ and $\hbold\in T_{\fboldstar}\Mrom$.  Note that $Q(\hbold, \hbold)=0$, since $\partial^2_{\epsilon}\Phi(c\fboldstar+\epsilon \hbold)|_{\epsilon=0}$ vanishes in that case, as we are differentiating along some curve associated to some symmetry of~\eqref{eq:BezRogers}, on which $\Phi$ is constant. 
    
    Since $\fboldstar$ is a global minimiser of $\Phi$, in particular $Q$ is a positive-semidefinite quadratic form. We can thus apply the Cauchy-Schwarz inequality to obtain that
    \begin{equation}\label{eq:Cauchy-Schwarz_toQ}
        \lvert Q(\gbold, \hbold)\rvert\le Q(\gbold, \gbold)Q(\hbold, \hbold)=0. 
    \end{equation}
    We conclude that 
    \begin{equation}\label{eq:lemma_proof_concl}
        Q(\fboldbot, \fboldbot)=Q(\gbold, \gbold)\ge C_\star\lVert \gbold\rVert_{\mathcal H}^2=C_\star\lVert\fboldbot\rVert_{\mathcal H}^2+C_\star\lVert \hbold\rVert_{\mathcal H}^2\ge C_\star\lVert\fboldbot\rVert_{\mathcal H}^2.
    \end{equation}
\end{proof}

From now on we will be working on the sphere $\SSS^5\subset \R^6$, whose points we denote by $X=(X_0, \vec{X})$, with $\vec{X}=(X_1, \ldots, X_5)$; thus, $\sum_0^5 X_j^2=1$. We let $d\sigma$ denote the standard hypersurface measure on $\SSS^5$. 
\begin{definition}[Penrose transform] \label{def:penrose} We identify $\fbold=(f_0, f_1)\in \mathcal H$ with the pair $(F_0, F_1)$ of real functions on $\SSS^5$ via the formulas
\begin{equation}\label{eq:Penrose_Transform}
    \begin{array}{ccc}
         f_0(x)=(1+X_0)^2F_0(X), & f_1(x)=(1+X_0)^3 F_1(X), & \text{where }x=\frac{\vec{X}}{1+X_0}.
    \end{array}
\end{equation}
\end{definition}
\begin{remark}\label{rem:penrose_point}
The map $(X_0, \vec{X})\mapsto x=\frac{\vec{X}}{1+X_0}$ is the stereographic projection of $\mathbb S^5\setminus\{(-1, \vec{0})\}$ onto $\mathbb R^5$. The identification~\eqref{eq:Penrose_Transform} implies (see, for example,~\cite[eq.~(19)-(20)]{Ne18})
\begin{equation}\label{eq:penrose_evolution}
    \begin{array}{cc}
         S\fbold(t, r\omega)=\cos(T\sqrt{4-\Delta_{\SSS^5}})F_0(X) + \frac{\sin(T\sqrt{4-\Delta_{\SSS^5}})}{\sqrt{4-\Delta_{\SSS^5}}}F_1(X), & \text{where }X=(\cos R, \sin R\omega), 
    \end{array}
\end{equation}
for $r\ge 0$ and $\omega\in \mathbb S^4$. The variables $(T, X)\in [-\pi, \pi]\times \mathbb S^5$ are related to $(t, r\omega)\in \mathbb R^{1+5}$ via the formulas $T=\arctan(t+r)+\arctan(t-r), R=\arctan(t+r)-\arctan(t-r)$, which, however, we will not need in the sequel.
\end{remark}
Under~\eqref{eq:Penrose_Transform}, $\fboldstar$ corresponds to the pair of constant functions $F_{\!\star\,0}=1, F_{\!\star\,1}=0$, while $T_\fboldstar\Mrom$ corresponds to the following space of polynomials in $(X_0, X_1, \ldots, X_5)$;
\begin{equation}\label{eq:tangent_space_dim_five}
	T_{\fboldstar}\Mrom\equiv\Set{ \begin{bmatrix} (1+X_0)^2\ (\sum_{j=0}^5 a_j X_j + a_6) \\ (1+X_0)^3\ ( b_0X_0 + b_1) \end{bmatrix}  : a_j, b_j\in \R}.
\end{equation}
This can be seen by redoing verbatim the computations in~\cite[Section~3]{Ne18}.

\begin{remark}\label{rem:smaller_tangspace}
    Note that~\eqref{eq:tangent_space_dim_five} is strictly smaller than the analogous tangent space~\cite[eq.~(34)]{Ne18} in the conformal case. This is because the present case has less symmetries, as we saw.
\end{remark}

We now introduce the system of spherical harmonics on $\SSS^5$ which we will use. This fine description of spherical harmonics was not needed in the aforementioned~\cite{Ne18}.
\begin{definition}\label{def:sph_harm} Let $\Set{ Y_{\ell, \mbold}}$ denote a complete orthonormal system of $L^2(\SSS^5)$, fixed once and for all according to the following prescription. The indices range on 
\begin{equation}\label{eq:QuantumNumbersFive}
    \begin{array}{cc}
	    \ell\in\mathbb N_{\ge 0}, &\mbold\in\QuantumN(\ell):=\{ \mbold=(m_1, m_2, m_3, m_4)\in \Z^4\ :\ \ell \ge m_1\ge m_2 \ge m_3 \ge |m_4|\},
    \end{array}
\end{equation}
and each $Y_{\ell, \mbold}$ has the form
\begin{equation}\label{eq:structure_spherical_harm}
    Y_{\ell, \mbold}(X_0, \vec{X})=P_\ell^{m_1}(X_0) Y_{\mbold}^{\SSS^4}(\vec{X}/\lvert \vec X \rvert),
\end{equation}
with $P_\ell^{m_1}$ denoting the $6$-dimensional normalized associated Legendre function (see the appendix), while $\{Y^{\SSS^4}_{\mbold}\}_{\mbold\in \QuantumN(\ell)}$ denotes a orthonormal system of real spherical harmonics in $L^2(\SSS^4)$ of degree $m_1$.\footnote{In particular, each $Y_{\ell, \mbold}$ is a spherical harmonic on $\SSS^5$ of degree $\ell$ and such that $\lVert Y_{\ell, \mbold}\rVert_{L^2(\mathbb S^5)}=1$; see~\cite[Lemma~1, p.~55]{Muller98}.}  For each $F\in L^2(\SSS^5)$, let $\Fhat(\ell, \mbold):=\int_{\SSS^5} F Y_{\ell,\mbold}\, d\sigma$.
\end{definition}
With these definitions, we can characterize~\eqref{eq:tangent_space_dim_five} as
\begin{equation}\label{eq:belong_tangent_fived}
	\begin{array}{ccc}
		\fbold\in T_{\fboldstar} \Mrom & \iff & \begin{cases}	
			\Fhat_0(\ell, \mbold)=0, & \ell\ge 2,\\ 
			\Fhat_1(\ell,\mbold)=0, & \ell \ge 2\text{ or }\ell=1,\ \mbold\ne \obold,
			\end{cases}
		\end{array}
	\end{equation}
which suggests the introduction of the following alternative orthogonality relation, towards the application of Lemma~\ref{lem:change_ortho}:
\begin{equation}\label{eq:TildeBot}
	\begin{array}{ccc}
	\gbold\widetilde{\bot} T_{\fboldstar}\Mrom & \overset{\mathrm{def}}{\iff} & \begin{cases} \Ghat_0(\ell, \mbold)=0, \\ \Ghat_1(\ell, \obold)=0,\end{cases}\, \text{for } \ell=0,1,\ \mbold\in\QuantumN(\ell).
	\end{array}
\end{equation}
Here, $(F_0, F_1)$ and $(G_0, G_1)$ denote the Penrose transforms~\eqref{eq:Penrose_Transform} of $\fbold$ and $\gbold$ respectively. Note that $\widetilde{\bot}$ is different from the standard orthogonality $\bot$ of $\mathcal H$; see the appendix. 

\begin{remark}\label{rem:cumbersome_because_not_conformal}
    In the conformal case of~\cite{Ne18}, there is no need to introduce such alternative orthogonality relations. Indeed, the natural scalar product considered there is diagonalized by the spherical harmonics, after the Penrose transform; see~\cite[eq.~(24)]{Ne18}. 
\end{remark}

For $\fbold=(f_0, f_1)\in\mathcal{H}$, a Taylor expansion to second order of the deficit functional~\eqref{eq:deficit_funcional_def} shows that the quadratic form $Q(\fbold, \fbold)$ equals
\begin{equation}\label{eq:general_quadratic_form_preliminary}
    \begin{split}
        &\frac1{4\pi}\lVert \fbold\rVert_{\mathcal H}^2 - 6\lVert S\fboldstar\rVert_{L^4(\mathbb R^{1+5})}^{-2}\int_{\mathbb R^{1+5}} (S\fboldstar)^2(S \fbold)^2 + 4 \lVert S\fboldstar\rVert_{L^4(\mathbb R^{1+5})}^{-6}\left( \int_{\mathbb R^{1+5}} (S\fboldstar)^3 S\fbold\right)^2. \\
    \end{split}
\end{equation}
We already observed that $\Phi(\fboldstar)=0=\partial_{\epsilon}\Phi(\fboldstar + \epsilon \fbold)|_{\epsilon=0}$. This  yields
\begin{equation}\label{eq:simplify_quadform}
    \begin{array}{cc}
        \displaystyle\lVert S\fboldstar \rVert_{L^4(\mathbb R^{1+5})}^2=\frac{1}{8\pi}\lVert \fboldstar\rVert_{\mathcal H}^2, &\displaystyle \int_{\mathbb R^{1+5}}(S\fboldstar)^3 S\fbold = \frac{\lVert S\fboldstar\rVert_{L^4(\mathbb R^{1+5})}^4}{\lVert \fboldstar \rVert_{\mathcal H}^2} \Braket{ \fboldstar | \fbold}_{\mathcal H},
    \end{array}
\end{equation}
which we can insert into~\eqref{eq:general_quadratic_form_preliminary} to obtain the simpler expression 
\begin{equation}\label{eq:general_quadratic_form_fived}
    Q(\fbold, \fbold)=\frac{16\pi}{\lVert\fboldstar\rVert_{\mathcal H}^2}\left[ \frac{1}{(8\pi)^2}\big(2(\Braket{\fboldstar|\fbold}_{\mathcal H})^2+\lVert\fboldstar\rVert_{\mathcal H}^2\lVert\fbold\rVert_{\mathcal H}^2\big) -3\int_{\mathbb R^{1+5}}(S\fboldstar)^2(S \fbold)^2\right].
\end{equation}
It is easy to see that $Q(\fbold, \fbold)=Q((f_0, 0), (f_0,0))+Q((0,f_1), (0, f_1))$ (see~\cite[eq.~(50)]{Ne18}). The computations in the subsection~\ref{sec:appendix_quadforms} of the appendix show that, for all $\fbold\widetilde{\bot}T_{\fboldstar}\Mrom$,
\begin{equation}\label{eq:Quad_Fived_zero}
    Q((f_0, 0), (f_0, 0))= 
    \frac1{4\pi}\Big[\sum_{\ell=2}^\infty\sum_{\mbold \in \QuantumN(\ell)}\!\! \alpha_{\ell, \mbold}\Fhat_0(\ell,\mbold)^2 
     \!+\!\beta_{\ell, \mbold}\Fhat_0(\ell+1,\mbold)\Fhat_0(\ell,\mbold) \Big],
\end{equation}
while 
\begin{equation}\label{eq:Quad_Fived_one}
\begin{split} 
Q((0, f_1),(0, f_1))=\,
	\frac1{4\pi}\Big[&\sum_{\mbold\in\QuantumN(1), m_1=1} 3\,\frac{\Fhat_1(1,\mbold)^2}{9}  \\ 
    & +\sum_{\ell=2}^\infty\sum_{\mbold \in \QuantumN(\ell)}\!\! \alpha_{\ell, \mbold}\frac{\Fhat_1(\ell,\mbold)^2}{(\ell+2)^2}  
    \!+\! \beta_{\ell, \mbold}\frac{\Fhat_1(\ell,\mbold)\Fhat_1(\ell+1,\mbold)}{(\ell+2)(\ell+3)}\Big],
   \end{split}
  \end{equation}
where the coefficients are
%
%
\begin{equation}\label{eq:alpha_beta}
    \begin{array}{cc}
        \alpha_{\ell, \mbold}=\frac{\ell^4+8\ell^3+11\ell^2-20\ell-12+6m_1^2+18m_1}{(\ell+1)(\ell+3)}, &      \beta_{\ell, \mbold}=(\ell-1)(\ell+6)\sqrt{\frac{ (\ell +1 -m_1)(\ell+4 +m_1)}{(\ell+2)(\ell+3)}}.
    \end{array}
\end{equation}
These formulas show that $Q$ has a kind of tridiagonal structure; for example, neglecting all summands with $\mbold\ne \obold$, we can formally write~\eqref{eq:Quad_Fived_zero} as
\begin{equation}\label{eq:infinite_tridiagonal}
    \begin{bmatrix} \Fhat_0(2, \obold) & \Fhat_0(3, \obold) & \Fhat_0(4, \obold) &\ldots\end{bmatrix} 
    \begin{bmatrix} 
        \alpha_{2, \obold} & \tfrac12\beta_{2, \obold} & 0 &  0 & \\ 
        \tfrac12\beta_{2, \obold} & \alpha_{3, \obold} & \tfrac12\beta_{3, \obold} & 0&    \\
         0& \tfrac12\beta_{3, \obold} & \alpha_{4, \obold} & \tfrac12\beta_{4, \obold} &   \\ 
           & & \ddots & \ddots & \ddots \\
    \end{bmatrix}
    \begin{bmatrix} \Fhat_0(2, \obold) \\  \Fhat_0(3, \obold) \\ \Fhat_0(4, \obold) \\\vdots\end{bmatrix}.
\end{equation}
\begin{remark}\label{rem:ConformalDiagonalQuadratic} 
    The analogous quadratic form~\cite[eq.~(52)]{Ne18} for the conformal case is diagonal.
\end{remark}

In order to exploit such tridiagonal structure, we introduce the following criterion. 
  \begin{lemma}[Diagonal dominance] \label{lem:tridiag_dominance}
Let $L\in \N_{\ge 0}$ and let 
\begin{equation}\label{eq:quad_coeff}
	\Set{ a_{\ell, \mbold}, b_{\ell,\mbold} : \ell\in \N_{\ge L},\ \mbold\in\QuantumN(\ell)}
\end{equation}
 be real sequences satisfying
\begin{equation}\label{eq:tridiagonal_dominance_condition}
	\begin{cases}
		a_{L, \mbold}\ge \frac12 \abs{b_{L,\mbold} } , & \\
		a_{\ell,\mbold}\ge \frac12\Round{\abs{ b_{\ell,\mbold} }+ \abs{b_{\ell-1,\mbold}}}, & \ell>L,
	\end{cases}
\end{equation}
Then the quadratic functional $T$, defined by
\begin{equation}\label{eq:quad_form_spherical}
	T(F)=\sum_{\ell=L}^\infty \sum_{\mbold\in\QuantumN(\ell)} a_{\ell,\mbold} \Fhat(\ell,\mbold)^2 +b_{\ell,\mbold}\Fhat(\ell,\mbold)\Fhat(\ell+1,\mbold),
\end{equation}
satisfies $T(F)\ge 0$ for all $F\in L^2(\SSS^5)$.
\end{lemma}
\begin{proof}
With the convention that $b_{\ell,\mbold}=0$ if $\ell<L$ or $\ell<m_1$, we can bound $T(F)$ from below by 
\begin{equation}\label{eq:proof_tridiag_dominance}
	\begin{split}
		T(F)&\ge \sum_{\underset{\mbold\in\QuantumN(\ell)}{\ell\ge L}} \frac{\abs{b_{\ell,\mbold}}}{2} \Fhat(\ell,\mbold)^2 +\frac{\abs{b_{\ell-1,\mbold}}}{2}\Fhat(\ell,\mbold)^2 + b_{\ell,\mbold}\Fhat(\ell,\mbold)\Fhat(\ell+1,\mbold) \\
		&\ge \sum_{\underset{\mbold\in\QuantumN(\ell)}{\ell\ge L}}\!\!\frac{1}{2}\abs{b_{\ell,\mbold}} \Round{ \Fhat(\ell,\mbold) +\mathrm{sign}(b_{\ell, \mbold} )\Fhat(\ell+1,\mbold)}^2 \ge 0.
	\end{split}
\end{equation}
\end{proof}
We can finally apply this lemma to obtain the desired lower bound. Recall that the relation $\widetilde{\bot}$ has been defined in \eqref{eq:TildeBot}.
\begin{proposition}\label{prop:Qgzero_positive}
For all $\fbold\, \widetilde{\bot}\, T_\fboldstar \Mrom$,
\begin{equation}\label{eq:Qg_positive}
		Q(\fbold, \fbold)\ge \frac{36}{85}\frac{1}{8\pi}\norm{\fbold}_{\mathcal H}^2.
\end{equation}
\end{proposition}
Once Proposition~\ref{prop:Qgzero_positive} is proved, Lemma~\ref{lem:change_ortho} will imply the same lower bound with the standard orthogonality $\bot$ instead of $\widetilde{\bot}$, thus establishing the required coercivity of $Q$ and completing the proof of Theorem~\ref{thm:main_local}, hence of Theorem~\ref{thm:main_fived}.
\begin{proof}[Proof of Proposition~\ref{prop:Qgzero_positive}]
We observed that $Q( \fbold, \fbold)=Q((f_0,0), (f_0,0))+Q((0, f_1), (0, f_1))$. We consider the term $Q((f_0, 0), (f_0, 0))$ first. Defining the quadratic functional 
\begin{equation}\label{eq:T_op}
	\begin{array}{c}
	\displaystyle	T\colon \Set{ \Fhat_0(\ell, \mbold)=0,\  \text{for}\ \ell=0, \ell=1,\ \mbold\in \QuantumN(\ell)} \to\mathbb R, \\ \displaystyle T(f_0):=Q((f_0, 0), (f_0, 0))- \frac{36}{85}\frac{1}{8\pi}\norm{(f_0, 0)}_{\mathcal{H}}^2,
	\end{array}
\end{equation}
it will suffice to show that $T$ satisfies the conditions of Lemma~\ref{lem:tridiag_dominance}; notice that the orthogonality $(f_0, 0)\widetilde{\bot}T_\fboldstar\Mrom$ is encoded in the domain of $T$. We perform the change of variable 
\begin{equation}\label{eq:FromFtoG}
	\Fhat_0(\ell,\mbold)=\frac{\Hhat(\ell,\mbold)}{\sqrt{(\ell+1)(\ell+3)}},
\end{equation}
so that, by~\eqref{eq:Quad_Fived_zero} and by the formula~\eqref{eq:Hone_stereographic} in the appendix, we have
\begin{equation}\label{eq:PrepareTridiagonalApplication}
	T(H)=\sum_{\ell=2}^\infty\sum_{\mbold\in\QuantumN(\ell)} a_{\ell,\mbold} \Hhat(\ell,\mbold)^2 +b_{\ell,\mbold}\Hhat(\ell,\mbold)\Hhat(\ell+1,\mbold),
\end{equation}
where
\begin{equation}\label{eq:a_and_b_of_tridiagonal_application}
	\begin{array}{rcl}
		a_{\ell, \mbold} &=& \frac1{4\pi}\frac{\ell^4+8\ell^3+11\ell^2-20\ell -12 +6m_1^2+18m_1}{  (\ell+1)^2(\ell+3)^2} -  \frac{36}{85}\frac{1}{8\pi}\frac{ (\ell+2)^2}{(\ell+1)(\ell+3)}, \\ 
		b_{\ell,\mbold} &=& \sqrt{ \frac{ (\ell+1-m_1)(\ell+4+m_1)}{(\ell+1)(\ell+4)}}\Round{\frac1{4\pi}\frac{(\ell-1)(\ell+6)}{(\ell+2)(\ell+3)} - \frac{36}{85}\frac{1}{8\pi}}.
	\end{array}
\end{equation}
Notice that $b_{\ell, \obold}$ is a rational function: the change of variable \eqref{eq:FromFtoG} was chosen to obtain this. Note also that, for all $\ell\ge 2$, $a_{\ell,\mbold}\ge a_{\ell,\obold}$ and $b_{\ell,\mbold}\le b_{\ell,\obold}$. Therefore
\begin{equation}\label{eq:Qfive_minus_C_diag_dominant_stepone}
	a_{2,\mbold}-\frac12 b_{2,\mbold} \ge a_{2,\obold}-\frac12 b_{2, \obold} = 0,
\end{equation}
while, for $\ell>2$, 
\begin{equation}\label{eq:Qfive_minus_C_diag_dominant_steptwo}
	\begin{split}
	a_{\ell,\mbold}-\frac12  \Round{ b_{\ell,\mbold}+  b_{\ell-1,\mbold}}  &\ge a_{\ell,\obold}-\frac12\Round{ b_{\ell,\obold}+b_{\ell-1,\obold}} \\
	&= \frac1{4\pi(\ell+1)(\ell+3)} \left(\frac{\ell^2+4\ell+15}{(\ell+1)(\ell+3)}-\frac{18}{85}\right) >0.
\end{split}
\end{equation}
So the conditions \eqref{eq:tridiagonal_dominance_condition} of Lemma \ref{lem:tridiag_dominance} are satisfied, and we can conclude that 
 \begin{equation}\label{eq:Qgzero_positive}
 	\begin{array}{cc}
  	\displaystyle Q((f_0,0), (f_0,0))\ge \frac{36}{85}\frac{1}{8\pi} \norm{(f_0, 0)}_{\mathcal{H}}^2, & \text{for all }(f_0, 0)\widetilde{\bot} T_\fboldstar\Mrom.
	\end{array}
\end{equation}

To prove the analogous inequality for $Q((0, g_1), (0, g_1))$ we let 
  \begin{equation}\label{eq:PrepTridOne}
  	\begin{array}{c}
	\displaystyle T\colon \Set{ \Fhat_1(\ell, \obold) = 0,\  \text{for}\  \ell=0,\ \ell=1} \to \mathbb R \\ \displaystyle 
	T(f_1):= Q((0, f_1), (0, f_1)) - \frac{36}{85}\frac1{8\pi} \norm{(0, f_1)}_\mathcal{H}^2 ,
	\end{array}
\end{equation}
We perform the change of variable 
  \begin{equation}\label{eq:FromGtoH}
  	\Fhat_1(\ell, \mbold)=\frac{\Hhat(\ell, \mbold)(\ell+2)}{\sqrt{ (\ell+1)(\ell+3)}},
\end{equation} 
so that, by~\eqref{eq:Quad_Fived_one} and by the formula~\eqref{eq:Ltwo_stereographic} in the appendix, 
\begin{equation}\label{eq:T_Op_One}
	\begin{split}
	T(H)&=\sum_{\mbold \in \QuantumN(1), m_1=1} \tilde{a}_{1, \mbold} \Hhat(1, \mbold)^2 + \tilde{b}_{1, \mbold}\Hhat(1, \mbold)\Hhat(2, \mbold) \\
		&+\sum_{\ell=2}^\infty\sum_{\mbold\in\QuantumN(\ell)} a_{\ell,\mbold} \Hhat(\ell,\mbold)^2 +b_{\ell,\mbold}\Hhat(\ell,\mbold)\Hhat(\ell+1,\mbold),
	\end{split}
\end{equation}
%
%
where $\tilde{a}_{1, \mbold}=\frac {93}{2720\pi} , \tilde{b}_{1, \mbold}= -\frac{36}{85}\frac{1}{8\pi}\sqrt{\frac35}$, while $a_{\ell, \mbold}$ and $b_{\ell, \mbold}$ equal \eqref{eq:a_and_b_of_tridiagonal_application} for $\ell\ge 2$. For $\ell=1, 2$ and $m_1=1$ we have that 
\begin{equation}\label{eq:a_one_check}
	\begin{array}{rcl}
	\tilde{a}_{1, \mbold}-\frac12 \lvert \tilde{b}_{1, \mbold}\rvert & =& \frac{1}{\pi}\Round{{\frac {93}{2720}} -\frac{36}{85}\frac{1}{8}\sqrt{\frac35}}>\frac1{100} >0,  \\
	a_{2, \mbold} - \frac12 \Round{ b_{2, \mbold} +\lvert\tilde{b}_{1, \mbold}\rvert} &=&\frac1\pi\left( {\frac {64}{1275}}-{\frac {2\sqrt{7}}{255}} -{\frac {9\sqrt{15}}{1700
}}\right)  >\frac2{1000}>0.
	\end{array}
\end{equation}
For all other values of $\ell$ and $\mbold$, the assumptions of Lemma \ref{lem:tridiag_dominance} have already been verified; see~\eqref{eq:Qfive_minus_C_diag_dominant_stepone} for the $\ell=2, m_1=2$ case (recall that, by convention, $b_{1, \mbold}=0$ if $m_1>1$), and~\eqref{eq:Qfive_minus_C_diag_dominant_steptwo} for all the other cases. We conclude that
\begin{equation}\label{eq:Qgone_positive}
 	\begin{array}{cc}
  	\displaystyle Q((0,f_1), (0,f_1))\ge \frac{36}{85}\frac{1}{8\pi} \norm{(0, f_1)}_{\mathcal{H}}^2, & \text{for all }(0, f_1)\widetilde{\bot} T_\fboldstar\Mrom,
	\end{array}
\end{equation}
which completes the proof.
\end{proof}

\begin{remark}\label{rem:final}
    The same proof shows that $C=\frac{36}{85}\frac{1}{8\pi}$ is the largest constant such that the quadratic form $Q(\fbold, \fbold)-C\norm{\fbold}_{\mathcal H}^2$ is diagonally dominant in the sense of Lemma~\ref{lem:tridiag_dominance}. This is the reason why the constant $\frac{36}{85}\frac{1}{8\pi}$ appears in Theorem~\ref{thm:main_local}.
\end{remark}

\appendix

\section{Computations with spherical harmonics}\label{sec:assoc_legendre}
\setcounter{equation}{0}
\setcounter{theorem}{0}
\renewcommand{\theequation}{A.\arabic{equation}}
\renewcommand{\thetheorem}{A.\arabic{theorem}}
In this appendix we compute expressions for the scalar product $\Braket{\fbold|\gbold}_{\mathcal H}$ and the quadratic form $Q(\fbold, \fbold)$ in terms of the Penrose transforms $(F_0, F_1)$ and $(G_0, G_1)$ of $\fbold$ and $\gbold$ respectively (see Definition~\ref{def:penrose}). 

Following \cite[pp.\,54]{Muller98}, we introduce the \emph{normalized associated Legendre functions} of  degree $\ell\in \N_{\ge 0}$, order $m\in \Set{0, 1, \ldots, \ell}$ and dimension $6$ to be the functions
\footnote{In terms of Gegenbauer polynomials, given via the generating function $\sum_{\ell=0}^\infty C_\ell^{(\nu)}(t)r^\ell = (1-2rt+r^2)^{-\nu}$, it holds that $P_\ell^m(6; t)=\mathcal{N}_{\ell, m}(1-t^2)^\frac{m}2 C_{\ell-m}^{(m+2)}(t)/C_{\ell-m}^{(m+2)}(1)$.}
\begin{equation}\label{eq:associated_legendre}
	\begin{array}{cc}
	P_\ell^m(6 ; t)=\mathcal{N}_{\ell, m}(1-t^2)^\frac m 2 P_{\ell-m}(2m+6;t), & t\in[-1,1],
	\end{array}
\end{equation}
where $P_{\ell-m}(2m+6;\cdot)$ is the \emph{Legendre polynomial} of degree $\ell-m$ in dimension $2m+6$.   The normalization constant (recall $\lvert\mathbb S^n\rvert=2\pi^\frac{n+1}{2}/ \Gamma(\frac{n+1}{2})$)
\begin{equation}\label{eq:norm_cnst}
	\mathcal{N}_{\ell, m}=\sqrt{\frac{(2\ell+4)(\ell+m+3)!}{(\ell-m)!(2m+4)!} \frac{\abs{ \SSS^{2m+4} } }{ \abs{\SSS^{2m+5}} } }
\end{equation} 
is chosen to ensure the orthonormality
\begin{equation}\label{eq:NormAssoLegendre}
	\int_{-1}^1 P_\ell^m(6; t)P^m_{\ell'}(6; t) (1-t^2)^\frac{3}{2}\, dt = \delta_{\ell,\ell'}.
\end{equation}
We adopt the convention that $P_\ell^m(6; \cdot)=0$ if $m>\ell$. 

Recall from Definition~\ref{def:sph_harm} that \begin{equation}\label{eq:Yell_Recall}
    Y_{\ell, \mbold}(X_0, \vec X)=P_\ell^{m_1}(6; X_0)Y_{\mbold}^{\mathbb S^4}\left(\frac{\vec{X}}{\lvert \vec X\rvert}\right), 
\end{equation}
where $\{Y_{\mbold}^{\mathbb S^4}\}$ is a fixed orthonormal system of spherical harmonics  on $\mathbb S^4$ of degree $m_1$; here
\begin{equation}\label{eq:sph_harm_recall}
    \begin{array}{cc}
	    \ell\in\mathbb N_{\ge 0}, &\mbold\in\QuantumN(\ell):=\{ \mbold=(m_1, m_2, m_3, m_4)\in \Z^4\ :\ \ell \ge m_1\ge m_2 \ge m_3 \ge |m_4|\}.
    \end{array}
\end{equation}
Note that $\QuantumN(\ell)\subset\QuantumN(\ell+1)$. We now introduce the following coefficient, defined for $\ell, m_1\in \Z$:
	\begin{equation}\label{eq:SFact}
		\mathrm{C}_5(\ell, m_1)=
		\begin{cases} 
		    \frac12\sqrt{ \frac{ (\ell-m_1+1)(\ell+m_1+4) }{(\ell+2)(\ell+3)}} , & 0\le m_1\le \ell, \\ 
		    0 &\text{otherwise}.
		\end{cases}
	\end{equation}
This appears in the next lemma.
\begin{lemma}\label{lem:IntSphericalProduct}
	For all $\ell, \ell'\in \N_{\ge 0}$ and all $\mbold\in \QuantumN(\ell), \mbold'\in \QuantumN(\ell')$, 
	\begin{equation}\label{eq:integral_product}
	\int_{\SSS^5} X_0 Y_{\ell, \mbold}(X) Y_{\ell',\mbold'}(X)\, d\sigma = 
	\begin{cases}  
	    \mathrm{C}_5(\min(\ell, \ell'), m_1), & \lvert \ell'-\ell\rvert=1 \text{ and } \mbold=\mbold',\\ 
	    0, & \text{otherwise}.
	\end{cases}
\end{equation}
\end{lemma}
\begin{proof}
Letting $X_0=\cos R$, we have that $d\sigma=(\sin R)^4\, dR\, d\sigma_{\mathbb S^4}$; thus, 
\begin{equation}\label{eq:towards_m_is_mprime}
    \begin{split}
        \int_{\mathbb S^5}X_0 Y_{\ell, \mbold} Y_{\ell', \mbold'}\, d\sigma &= \int_{\mathbb S^4} Y_{\mbold}^{\mathbb S^4} Y_{\mbold'}^{\mathbb S^4}\, d\sigma_{\mathbb S^4}\int_0^\pi \cos(R)\, P_{\ell}^{m_1}(6; \cos R)\,P_{\ell'}^{m_1'}(6;\cos R)(\sin R)^4\, dR \\ 
        &=\delta_{\mbold, \mbold'}\int_{-1}^1 X_0P_\ell^{m_1}(6;X_0)P_{\ell'}^{m_1}(6;X_0)(1-X_0^2)^\frac32\, dX_0.
    \end{split}
\end{equation}
To evaluate the latter integral, we first assume without loss of generality $\ell'\ge \ell$.  From the recursion relation for the Legendre polynomials (\cite[Lemma 3, pg.\,39]{Muller98}) we obtain
\begin{equation}\label{eq:associated_legendre_recurrence}
		0=a_\ell^{m_1}P_\ell^{m_1}(6; X_0) - b_\ell^{m_1}X_0P_{\ell-1}^{m_1}(6;X_0) +c_\ell^{ m_1} P_{\ell-2}^{m_1}(6;X_0),
\end{equation}
with 
\begin{equation}\label{eq:recurrence_coeffs}
	\begin{array}{ccc}
		a_\ell^{m_1}= \sqrt{ \frac{(\ell-m_1)(\ell+m_1+3)}{(2\ell+4)(\ell+m_1+2)}}, &
		b_\ell^{m_1} = \sqrt{ \frac{2\ell+2}{\ell+m_1+2} }, &
		c_\ell^{m_1} = \sqrt{ \frac{\ell-m_1-1}{2\ell}}.
	\end{array}
\end{equation}
Multiplying~\eqref{eq:associated_legendre_recurrence} by $P_{\ell'-1}^{m_1}(6; X_0)(1-X_0^2)^\frac{3}{2}$ and then integrating, we infer from \eqref{eq:NormAssoLegendre} that, since $\ell'\ge \ell$,
\begin{equation}\label{eq:spherical_factor}
	\int_{-1}^1 P_{\ell-1}^{m_1}(6; X_0)P_{\ell'-1}^{m_1}(6; X_0)X_0(1-X_0^2)^{\frac{3}{2}}\,dX_0=\frac{a_\ell^{ m_1}}{b_\ell^{m_1}}\delta_{\ell, \ell'-1}=\mathrm{C}_5(\ell-1, m_1)\delta_{\ell, \ell'-1}.
\end{equation}
This completes the proof.
\end{proof}
In order to compute a convenient expression for $\Braket{\fbold|\gbold}_{\mathcal H}$, where as usual $\fbold=(f_0, f_1)$ and $\gbold=(g_0, g_1)$, we start by rewriting it in terms of the fractional Laplacian, as follows:  
\begin{equation}\label{eq:Hcaldotone_scalprod}
	\Braket{\fbold| \gbold}_{\mathcal{H}} = \int_{\R^5}\sqrt{-\Delta} f_0 \sqrt{-\Delta} g_0\, dx + \int_{\R^5} f_1 g_1\, dx .
\end{equation}
We identify $\fbold$ to $(F_0, F_1)$ and $\gbold$ to $(G_0, G_1)$ via the Penrose transform of Definition~\ref{def:penrose}, and we recall the fractional Laplacian formula (see, e.~g.,~\cite[Lemma~A.3]{NeOSStTa23})
\begin{equation}\label{eq:morpurgo}    
    \sqrt{-\Delta}f_0(x)=(1+X_0)^3\sqrt{4-\Delta_{\mathbb S^5}}(F_0)(X), 
\end{equation}
where $x$ and $X=(X_0, \vec{X})$ are related by the stereographic projection $x=\vec{X}/(1+X_0)$, as in Definition~\ref{def:penrose}. Obviously, the same formula holds for for $g_0$ and $G_0$. 

Recalling the Jacobian $dx=(1+X_0)^{-5}d\sigma$, we have that   
\begin{equation}\label{eq:stereographic_projection_aftermath}
	\Braket{\fbold| \gbold}_{\mathcal{H}} = \int_{\SSS^5} \sqrt{4-\Delta_{\mathbb S^5}}F_0\ \sqrt{4-\Delta_{\mathbb S^5}}G_0\ (1+X_0)\, d\sigma + \int_{\SSS^5} F_1G_1(1+X_0) d\sigma.
\end{equation}
We now use Lemma \ref{lem:IntSphericalProduct} to compute 
\begin{equation} \label{eq:Ltwo_stereographic}
    \begin{split}
		\int_{\SSS^5} F_1G_1(1+X_0) d\sigma &= 
		\sum_{\ell=0}^\infty\sum_{\mbold\in \QuantumN(\ell)}\Fhat_1(\ell, \mbold)\Ghat_1(\ell, \mbold) \\   &+\mathrm{C}_5(\ell, m_1) \Round{ \Fhat_1(\ell,\mbold)\Ghat_1(\ell+1, \mbold) + \Fhat_1(\ell+1, \mbold)\Ghat_1(\ell, \mbold)}.
	\end{split}
\end{equation}
Since $-\Delta_{\mathbb S^5}Y_{\ell, \mbold}=\ell(\ell+4)Y_{\ell, \mbold}$, we have that $\sqrt{4-\Delta_{\mathbb S^5}}Y_{\ell, \mbold}=(\ell+2)Y_{\ell,\mbold}$. Thus, $\int_{\SSS^5} \sqrt{4-\Delta_{\mathbb S^5}}F_0\ \sqrt{4-\Delta_{\mathbb S^5}}G_0\ (1+X_0)\, d\sigma$ is equal to
\begin{equation}\label{eq:Hone_stereographic}
    \begin{split}
	\sum_{\ell\ge 0}&\sum_{\mbold\in \QuantumN(\ell)} \Round{\ell+2}^2\Fhat_0(\ell, \mbold)\Ghat_0(\ell, \mbold)\\ &+\mathrm{C}_5(\ell, m_1)\Round{\ell+2}\Round{\ell+3}\Round{ \Fhat_0(\ell,\mbold)\Ghat_0(\ell+1, \mbold) + \Fhat_0(\ell+1, \mbold)\Ghat_0(\ell, \mbold)}. 
	\end{split}
\end{equation}
\begin{remark}\label{rem:cumbersome_orthogonality}
    These formulas show that $\Braket{\cdot|\cdot}_\mathcal{H}$ is not diagonal in $\Fhat_0(\ell, \mbold)$ and $\Fhat_1(\ell, \mbold)$. This is the reason why the orthogonality $\fbold\bot T_\fboldstar\Mrom$ is difficult to characterize in terms of these coefficients.
\end{remark}
As an application, we now compute $\lVert\fboldstar\rVert_{\mathcal H}^2$. The Penrose transform of $\fboldstar$ is the pair of constant functions $(F_{\!\star 0}, F_{\!\star 1})=(1, 0)$, so $\Fhat_{\!\star 0}(0, \obold)=\sqrt{\abs{\mathbb S^5}}$ and $\Fhat_{\!\star 0}(\ell, \mbold)=0$ for all $\ell>0$. We conclude from~\eqref{eq:Hone_stereographic} that
\begin{equation}\label{eq:norm_of_maximizer}
    \lVert\fboldstar\rVert_{\mathcal H}^2=4\abs{\mathbb S^5}=4\pi^3.
\end{equation}
\subsection{The quadratic form}\label{sec:appendix_quadforms}
In this subsection, we evaluate an expression for the quadratic form $Q(\fbold, \fbold)$; recall~\eqref{eq:general_quadratic_form_fived}. We need to compute it for $\fbold\widetilde{\bot}T_{\fboldstar}\Mrom$, that is, 
\begin{equation}\label{eq:FboldOrthoAppendix}
    \begin{array}{ccc}
        \Fhat_0(\ell, \mbold)=0, & \Fhat_1(\ell, \obold)=0, & \text{ for }\ell=0, 1\text{ and }\mbold\in \QuantumN(\ell), 
    \end{array}
\end{equation}
which, by~\eqref{eq:Hone_stereographic}, immediately imply $\Braket{\fboldstar| \fbold}_{\mathcal H}=0$. Thus,  
\begin{equation}\label{eq:Q_appendix}
	\begin{split}
		Q((f_0, 0), (f_0, 0))&= \frac{1}{4\pi}\norm{(f_0,0)}_{\mathcal H}^2
		 -\frac{12}{\pi^2} \int_{\R^{1+5}} \Round{ S\fboldstar}^2\Round{S(f_0,0) }^2,\\
		 Q((0, f_1), (0, f_1))&= \frac{1}{4\pi}\norm{(0, f_1)}_{\mathcal H}^2
		 -\frac{12}{\pi^2} \int_{\R^{1+5}} \Round{ S\fboldstar}^2\Round{S(0, f_1) }^2.
	\end{split}
\end{equation}
%

To evaluate the latter integrals, we first note that, by~\cite[Corollary~3.7]{Ne18},
\begin{equation}\label{eq:CorollarySpacetimeIntegration}
    \begin{split}
        &\int_{\mathbb R^{1+5}}(S\fboldstar)^2(S\fbold)^2 =\\ & \frac12 \int_{\mathbb S^1\times \mathbb S^5}\Big[ \cos(2T)\Big( \cos(\sqrt{4-\Delta_{\mathbb S^5}}T) F_0
        +\frac{\sin({\sqrt{4-\Delta_{\mathbb S^5}}}T)}{\sqrt{4-\Delta_{\mathbb S^5}}}F_1\Big)\Big]^2(\cos T+X_0)^2\, dT d\sigma;
    \end{split}
\end{equation}
moreover, by Lemma~\ref{lem:IntSphericalProduct} (recall the convention $Y_{\ell-1, \mbold}=0$ if $\ell-1<0$ or $\ell-1 < m_1$),
\begin{equation}\label{eq:OmegaYlj}
	(\cos T+X_0) Y_{\ell, \mbold} = \cos (T )Y_{\ell, \mbold} +\mathrm{C}_5(\ell-1, m_1)Y_{\ell-1, \mbold} + \mathrm{C}_5(\ell, m_1)Y_{\ell+1, \mbold}.
\end{equation}  

Let $A_\ell(T):=\cos(2T)\cos((\ell+2)T)$. We have that
\begin{equation}\label{eq:Mauvais_Term_first_step}
    \begin{split}
	    &\frac{12}{\pi^2} \int_{\R^{1+5}} \Round{ S\fboldstar}^2\Round{S(f_0, 0)}^2 = \frac{6}{\pi^2}\int_{\SSS^1\times \SSS^5} \Big[\sum_{\ell, \mbold}A_\ell(T)\Fhat_0(\ell,\mbold)(\cos T+ X_0)Y_{\ell, \mbold}(X)\Big]^2\, dT d\sigma \\ 
	    &
	        \begin{split}
	            =\frac{6}{\pi^2} \int_{\SSS^1\times \SSS^5}\Big[
	        \sum_{\ell, \mbold} \Big(
	            A_\ell(T)\cos(T) \Fhat_0(\ell, \mbold) &+ A_{\ell-1}(T) \mathrm{C}_5(\ell-1, m_1) \Fhat_0(\ell-1, \mbold) \\&+ A_{\ell+1}(T)\mathrm{C}_5(\ell, \mbold) \Fhat_0(\ell+1, \mbold)
	            \Big)
	       Y_{\ell, \mbold}\Big]^2\, dTd\sigma 
	       \end{split} \\
	       &\begin{split}
	        =\frac{3}{8\pi^2} \sum_{\ell, \mbold} \int_{-\pi}^\pi\Big[ &\big(\cos((\ell-1)T)+\cos((\ell+3)T)\big) \big( \Fhat_0(\ell, \mbold) +2\mathrm{C}_5(\ell-1, m_1)\Fhat_0(\ell-1, \mbold)\big) \\
	        &+\big(\cos((\ell+1)T)+\cos((\ell+5)T)\big) \big( \Fhat_0(\ell, \mbold) +2\mathrm{C}_5(\ell, m_1)\Fhat_0(\ell+1, \mbold)\big) \Big]^2\, dT.
	        \end{split}
    \end{split}
\end{equation}
By the orthogonality~\eqref{eq:FboldOrthoAppendix}, the sum runs on $\ell\ge 1$, so the four cosines in the latter integral are orthogonal on $[-\pi, \pi]$. Using this we evaluate the integral, and rearrange terms, to conclude that $\frac{12}{\pi^2} \int_{\R^{1+5}} \Round{ S\fboldstar}^2\Round{S(f_0, 0)}^2$ equals
\begin{equation}\label{eq:SpacetimeIntegralFzero}
    \begin{split}
    &\begin{split}\frac{3}{8\pi}\sum_{\overset{\ell\ge 2}{\mbold\in\QuantumN(\ell)}}\Big[&(4+8\mathrm{C}_5(\ell, m_1)^2+8\mathrm{C}_5(\ell-1, m_1)^2)\Fhat_0(\ell, \mbold)^2+ \\ 
        &16\mathrm{C}_5(\ell, m_1)\Fhat_0(\ell, \mbold)\Fhat_0(\ell+1, \mbold)\Big]= 
        \end{split}\\ 
     &\frac{3}{\pi}\sum_{\overset{\ell\ge 2}{\mbold\in\QuantumN(\ell)}}\left[ \frac{2\ell^2\!+\!8\ell\!-\!m_1^2\!-\!3m_1\!+\!4}{2(\ell+1)(\ell+3)}\Fhat_0(\ell,\mbold)^2+2\mathrm{C}_5(\ell, m_1)\Fhat_0(\ell, \mbold)\Fhat_0(\ell+1,\mbold)\right].
    \end{split}
\end{equation}
Inserting this, and~\eqref{eq:Hone_stereographic}, into~\eqref{eq:Q_appendix} yields the formula~\eqref{eq:Quad_Fived_zero} of the main text. 
%
%

To compute the other term, it is convenient to let 
\begin{equation}\label{eq:Ghat_Appendix} 
    \Ghat_1(\ell, \mbold):=\frac{\Fhat_1(\ell, \mbold)}{\ell +2}. 
\end{equation}
Arguing as before, we see that $\frac{12}{\pi^2} \int_{\mathbb R^{1+5}}(S\fboldstar)^2(S(0, f_1))^2 $ equals
\begin{equation}\label{eq:MauvaisTermFone}
    \begin{split}
        &\frac{6}{\pi^2}\int_{\SSS^1\times \SSS^5} \Big[\sum_{\ell, \mbold}\cos(2T) \sin((\ell+2)T)\Ghat_1(\ell,\mbold)(\cos T+ X_0)Y_{\ell, \mbold}(X)\Big]^2\, dT d\sigma =\\ 
        &\begin{split}
	        \frac{3}{8\pi^2} \sum_{\ell, \mbold} \int_{-\pi}^\pi&\Big[ (\sin((\ell-1)T)+\sin((\ell+3)T)) \big( \Ghat_1(\ell, \mbold) +2\mathrm{C}_5(\ell-1, m_1)\Ghat_1(\ell-1, \mbold)\big) \\
	         &+(\sin((\ell+1)T)+\sin((\ell+5)T)) \big( \Ghat_1(\ell, \mbold) +2\mathrm{C}_5(\ell, m_1)\Ghat_1(\ell+1, \mbold)\big) \Big]^2\, dT=
	        \end{split}\\
	        &
	        \begin{split}
	            &\frac{3}{\pi}\sum_{\overset{\ell\ge 2}{\mbold\in\QuantumN(\ell)}}\left[ \frac{2\ell^2\!+\!8\ell\!-\!m_1^2\!-\!3m_1\!+\!4}{2(\ell+1)(\ell+3)}\Ghat_1(\ell,\mbold)^2+2\mathrm{C}_5(\ell, m_1)\Ghat_1(\ell, \mbold)\Ghat_1(\ell+1,\mbold)\right] \\
	            &+\frac{3}{\pi}\sum_{\obold\ne\mbold\in\QuantumN(1)}\left[\frac12\Ghat_1(1, \mbold)^2 + 2 \mathrm{C}_5(1, 1)\Ghat_1(1, \mbold)\Ghat_1(2, \mbold)\right].
	        \end{split}
    \end{split}
\end{equation}
This is very similar to the right-hand side of~\eqref{eq:SpacetimeIntegralFzero}, however it has extra summands in $\ell=1$ and $\obold\ne\mbold\in\QuantumN(1)$, due to the orthogonality~\eqref{eq:FboldOrthoAppendix}; indeed, notice that the term $\Fhat_1(1, \mbold)=(\ell+2)\Ghat_1(\ell, \mbold)$ needs not vanish. Inserting this, together with the formula~\eqref{eq:Ltwo_stereographic}, into~\eqref{eq:Q_appendix} finally yields the formula~\eqref{eq:Quad_Fived_one} of the main text.

\end{document}